\documentclass[11pt, a4paper]{article}
\usepackage{amsmath,amsthm,amsfonts,amssymb}
\usepackage{latexsym}
\usepackage{graphicx}

\newtheorem{theorem}{Theorem}
\newtheorem{lemma}[theorem]{Lemma}

\newtheorem{definition}{Definition}

\long\def\delete#1{}

\def\AA{{\cal A}}
\def\BB{{\cal B}}
\def\B{{\rm B}}
\def\a{\alpha}
\def\b{\beta}
\def\g{\gamma}
\def\t{\tau}

\title{{\bf Hadwiger's conjecture for  3-arc graphs\thanks{Research supported by ARC Discovery Project DP120101081.}}}
\author{David R. Wood$^a$\qquad Guangjun Xu$^b$ \qquad  Sanming Zhou$^b$  \\ \\
{\small 
 $^a$School of Mathematical Sciences} \\
 {\small   Monash University,  Melbourne, Australia}  \\
 {\small   {\texttt {david.wood@monash.edu} }} \smallskip \\ 
{\small
$^b$Department of Mathematics and Statistics}\\
{\small The University of Melbourne,  Parkville, VIC 3010,   Australia}\\
{\small   {\texttt{\{gx,smzhou\}@ms.unimelb.edu.au}}}
}

\date{\today}

\begin{document}

\openup 0.5\jot\maketitle

\vspace{-1cm}

\smallskip
\begin{abstract}
  The 3-arc graph of a digraph $D$ is 
defined to have vertices the arcs of $D$ such that two arcs
$uv, xy$ are adjacent if and only if  $uv$ and $xy$
are distinct arcs of $D$ with $v\ne x$, $y\ne u$ and $u,x$   adjacent.
 We prove that   Hadwiger's conjecture holds for   3-arc graphs. 

 \medskip


\end{abstract}

\section{Introduction}

A graph  $H$
is a {\em minor} of a graph  $G$ if  a graph isomorphic to  $H$ can be obtained from  a subgraph of   $G$ by contracting
edges. An {\em $H$-minor} is a minor isomorphic to $H$.
\delete{If $H$ is a complete graph, we also say that $G$ contains a clique minor of size   $|H|$.}  
The {\em Hadwiger number}    $h(G)$ of   $G$ is the  maximum integer $k$   such that
  $G$ contains   a   $K_k$-minor, where $K_k$ is the complete graph with $k$ vertices.

In 1943, Hadwiger  \cite{Had43} posed the following conjecture, which is thought to be one of the most 
difficult and beautiful problems in graph theory: 

\medskip

  {\bf Hadwiger's  Conjecture.}   For every graph   $G$,  $h(G)\ge \chi(G)$. 

\medskip

  Hadwiger's  conjecture  has been proved for  graphs $G$  with  $\chi(G)\le 6$ \cite{RST93}, and is open for graphs
with   $\chi(G)\ge 7$.  This conjecture also holds for particular classes of  graphs, including powers of cycles \cite{ll07}, 
 proper circular arc graphs \cite{bc09},
   line graphs \cite{rs04}, and  quasi-line graphs \cite{co08}.   See \cite{toft96} for a survey.

In this paper we prove   Hadwiger's  conjecture for  a large family of  graphs. Such graphs are defined by means of a graph operator, called the 3-arc graph construction (see Definition \ref{3arc}), which bears some similarities with the line graph operator.
 This construction was first introduced  by Li, Praeger and Zhou \cite{Li-Praeger-Zhou98} in the study of  a family of arc-transitive graphs whose automorphism group contains a subgroup acting imprimitively on the vertex set. (A graph is {\em arc-transitive} if its automorphism group is transitive on the set of oriented edges.) It was used in classifying or characterizing certain families of arc-transitive graphs \cite{Gardiner-Praeger-Zhou99,MPZ,Li-Praeger-Zhou98,LZ,Zhou00c, Zhou99,  Zhou98}.  
Recently, various graph-theoretic properties of 3-arc graphs have been   investigated \cite{bmg, KXZ, KZ, XZ}.

The original 3-arc graph construction \cite{Li-Praeger-Zhou98}   was defined for  a   finite, undirected  
and loopless graph $G = (V(G), E(G))$.  
In $G$,  an \emph{arc}   is an ordered pair of adjacent vertices. Denote by  $A(G)$  the set of arcs of $G$.
For adjacent vertices $u, v$ of $G$, we use $uv$ to denote the arc from $u$ to $v$, \delete{$vu$ ($\ne uv$) the arc
from $v$ to $u$,} and $\{u, v\}$ the edge between $u$ and $v$.  We emphasise that  each edge of $G$
gives rise to two arcs in $A(G)$.
A \emph{3-arc} of $G$ is a 4-tuple of vertices $(v, u, x, y)$, possibly with $v = y$, such that both $(v,u,x)$ and $(u,x,y)$ are paths of $G$.   
 The 3-arc graph of $G$ is defined as follows:

\begin{definition}\label{3arc}
{\em \cite{Li-Praeger-Zhou98, Zhou99}  
Let $G$ be an undirected graph. The 3-arc graph of $G$, denoted by $X(G)$, has
vertex set $A(G)$ such that two vertices corresponding to   arcs $uv$ and $xy$ are adjacent if and only if $(v,u,x,y)$ is a $3$-arc of $G$.
}
\end{definition}

\delete{Note that in Definition \ref{3arc}, the condition  that  $(v,u,x,y)$ is a $3$-arc of $G$    ensuring the adjacency of  distinct 
arcs   $uv$ and $xy$ in $X(G)$,   is equivalent to that $uv$ and $xy$
are distinct arcs of $G$ with $v\ne x$, $y\ne u$ and $u,x$   adjacent in $G$.
 From this observation,}

 The    3-arc graph construction can be  generalised  for a {\em digraph}  $D=(V(D), A(D))$  as follows \cite{KXZ},  
where  $A(D)$ is a multiset  of ordered pairs   (namely, arcs)  of distinct vertices of $V(D)$.  Here   a digraph allows parallel arcs but not loops.

\begin{definition}\label{3arc-gen} 
{\em 
Let   $D=(V(D), A(D))$   be a   digraph. The 3-arc graph of $D$, denoted by $X(D)$, has
vertex set $A(D)$ such that two vertices corresponding to     arcs $uv$ and $xy$ are adjacent if and only if 
  $v\ne x$, $y\ne u$ and $u,x$   are adjacent.
}
\end{definition}

Let $D$ be the digraph obtained from an undirected graph $G$  by replacing each edge $\{x,y\}$ by 
two opposite arcs $xy$ and $yx$. Then, 
$X(D)=X(G)$.

   Knor,  Xu and Zhou \cite{KXZ} introduced the notion of 
   {\em $3$-arc colouring}  a digraph,  which can be defined as  a proper vertex-colouring of $X(D)$.  
The minimum  number of colours  in a  $3$-arc colouring of $D$ is called the {\em $3$-arc chromatic index}
 of $D$, and is   denoted  by $\chi_{3}'(D)$.  Then  $\chi(X(D)) = \chi_{3}'(D)$.

 The main result of this paper is the following:


      
\begin{theorem}
\label{th3}
Let $D$ be a digraph   without    loops.  Then  $h(X(D))\geq   \chi (X(D))$.  
\end{theorem}

    Note that in   the case of  the $3$-arc graph of 
			an undirected graph,    we have obtained  a much    simpler proof of Theorem \ref{th3}.

\delete{ 
\begin{remark}
\label{rem:conca}
{\em 
 (1) In the proof of Theorem \ref{th3}, we   assume that, for every pair of distinct  $u$ and $v$ of $D$,  there is at most one arc is outgoing from 
one of  them  to the other. That is,    $A_D\{u,v\} \subseteq \{uv, vu\}$.  Since all the  arcs
   from $u$ to $v$  can be assigned the same colour and deleting arc does not increase    $h(X(D))$. 
		
	(2) Theorem \ref{th3} holds for the $3$-arc graphs of both  an undirected graph (viewed as a bi-oriented graph) and  an oriented graph. 
	For the case of  the $3$-arc graph of   an undirected graph, a much simple proof of Theorem \ref{th3} exists.  
}
\end{remark}
 }

\section{Preliminaries}

		We need the following  notation.    Let  $D=(V(D), A(D))$  be a digraph.  
We denote  by $A_D\{x,y\}$  the set of arcs between  vertices $x$ and $y$,  and     by $A_D(x)$   the set of arcs outgoing from $x$.
Then  vertices $x$ and $y$ are adjacent if and only if  $A_D\{x,y\}\ne \emptyset$.  
When  $|A_D\{x,y\}|= 1$, we   misuse the notation $A_D\{x,y\}$ to indicate the arc between $x$ and $y$.
 An {\em in-neighbour} 
 (respectively,  {\em out-neighbour}) of a vertex $x$ of $D$ is a vertex $y$ such that $yx\in A(D)$   (respectively,   $xy\in A(D)$).
The set of all in-neighbours  (respectively,  out-neighbours) of $x$  is denoted by $N_D^-(x)$  (respectively,  $N_D^+(x)$).
The {\em in-degree}  $d_D^-(x)$ (respectively, {\em out-degree} $d_D^+(x)$)   is defined to be the number of  in-neighbours  (respectively,  out-neighbours) of $x$. 
A vertex $x$ is called a {\em sink}  if  $d_D^+(x)=0$.
A digraph is  {\em simple} if  $|A_D\{x,y\}|\le 1$ for all distinct vertices $x$ and $y$ of $D$. 
A {\em tournament} is a simple digraph  whose underlying undirected graph is  complete.

For an undirected graph $G$, 
 the degree of a vertex $v$ in   $G$ is denoted by $d_G(v)$, and the minimum degree of $G$ is denoted by $\d(G)$.   
   We  omit the subscript when there is no ambiguity.

		A $K_t$-minor  in $G$  can be thought of as $t$ connected subgraphs in $G$ that  are
		pairwise disjoint  such that there is at least one edge of $G$ between each pair of subgraphs.
	  Each such subgraph is called a {\em branch set}. 

\begin{lemma}
\label{le1}
Let $D$ be a tournament on   $n\ge 5$ vertices. Then   $h(X(D))\ge n$.
\end{lemma}
    
\begin{proof}   Since $D$ is a tournament,  $A\{x,y\}$ is  interpreted as  a single arc.  
 Denote $V(D)=\{x, v_0, v_1, \ldots, v_{n-2}\}$. 
 We now  construct a collection of  $n$
branch sets.  
For  $0\le i \le n-2$, let $B_i:= \{A\{x, v_i\},  A\{v_{i+1}, v_{i+2}\}\}$.  Let  
  $U := \{A\{v_{i}, v_{i+2}\} \mid  0 \le i \le n-2\}$, where all subscripts are taken modulo $n-1$. 
	Clearly, these branch sets are pairwise disjoint.
  
   Now we show that each branch set is connected.  Note that each $B_i$   induces      $K_2$ in $X(D)$. 
	Since  $A\{v_{i}, v_{i+2}\}$ is adjacent to   $A\{v_{i+1}, v_{i+3}\}$ in $X(D)$,
   $U$ induces   a subgraph that 
  contains an $(n-1)$-cycle passing through each element of  $U$.

	Next we show that these branch sets are pairwise adjacent. 
  For each pair of distinct $B_i, B_{j}$, if $j\ne i+1$ and  $j\ne i+2$,  then
   $B_i$ and $B_{j}$ are adjacent  since  $A\{v_{i+1}, v_{i+2}\}$ is adjacent to $A\{v_0, v_{j}\}$.  If $j=i+1$, then 
     $i\ne j+1$ and  $i\ne j+2$  because $n-1\ge 4$, so  $A\{v_0, v_{i}\}$  is adjacent to $A\{v_{j+1}, v_{j+2}\}$.
    If  $j=i+2$,    then $A\{v_{j+1}, v_{j+2}\}$ is adjacent to   $A\{v_{i+1}, v_{i+2}\}$   since 
      $\{v_{j+1}, v_{j+2}\} \cap \{v_{i+1}, v_{i+2}\}=\emptyset$.
  Thus, $B_i$ is adjacent to $B_{j}$ as well.
  Since $A\{v_0, v_i\} \in B_i$ is adjacent to $A\{v_{i+1}, v_{i+3}\} \in U$, each $B_i$ is adjacent to $U$.   
\end{proof}

 Let $v$ be a vertex of a  digraph $D$.    Let $A\subseteq A(v)$.  
 An arc $xy$  is said to be   {\em $A$-feasible}  if $vx\in A$, $y\ne v$ and  $(v,x,y)$ is a directed path. 
  A set  $A^{f}\subseteq A(D)$   is  {\em $A$-feasible} if each arc in  $A^{f}$   is  $A$-feasible and no two arcs in   $A^{f}$ 
	share a tail.
 An arc $xy$ of $D$ is  said to be   {\em  $A$-compatible}  if  $y\ne v$,   $A(v,x)\ne \emptyset$  and  $vx\notin A$. 
 A set  $A^{c}\subseteq A(D)$ is   {\em  $A$-compatible} if each arc in  $A^{c}$  is $A$-compatible.
 Note that  each feasible arc $xy$ is adjacent in $X(D)$ to each arc in  $A$  except  $vx$, and each compatible  arc $xy$ is adjacent to 
each  arc  in $A$. 
 
\delete{
 maximal set of  feasible arcs   such
  that   no two arcs of   $A^{f}$ share the same tail. That is, for each $vx\in A$, 
 if  $A(x)-\{xv\}\ne \emptyset$, then 
 $A^{f}$ contains  exactly one arc of $A(x)-\{xv\}$.  Clearly, $0\le |A^{f}|\le |A|\le |A(v)|$.
 }

Let  $A^{f}$   be an  $A$-feasible  set,  and   $A^{c}$ be  an  $A$-compatible set. 
An {\em $(A,A^{f},A^{c})$-net} of {\em size $p$}  is a $K_p$-minor in $X(D)$ using only arcs in $A\cup A^{f} \cup A^{c}$ such
  that $p:=|A|$  and each branch set has exactly one arc in $A$.
	An  $(A,A^{f},A^{c})$-net is  called a {\em net at $v$} if  $v$ is the common tail of all arcs in $A$.  
	It may happen that one of $A^{f}$ and $A^{c}$ is empty. The following lemma
 provides some sufficient  conditions for the existence of an $(A,A^{f},A^{c})$-net.

\begin{lemma}
\label{le2}
Let  $v$ be a vertex of a  digraph $D$.   Let $A\subseteq A(v)$ and  $p:=|A|$.  Let 
 $A^{f}$   be an $A$-feasible set. Let    $A^{c}$ be  an     $A$-compatible set. 
Then,  in the following cases,  $D$ contains an $(A,A^{f},A^{c})$-net. 
     \begin{itemize}
    \item[\rm (1)]   $p=1$;
    \item[\rm (2)]     $|A^{c}|\ge 1$  and $p=2$;
      \item[\rm (3)]    $|A^{f}|= 3$ and $p=3$;
     \item[\rm (4)]  $|A^{f}|\ge 1$ and $|A^{c}|\ge1$ and  $p=3$;
       \item[\rm (5)] $|A^{c}|\ge2$ and  $p=3$;
         \item[\rm (6)]  $|A^{f}|+|A^{c}|\ge p-1$ and  $p\ge 4$.  
     \end{itemize}
\end{lemma}
  
\begin{proof} 
Denote $A =\{vv_0, vv_1, \ldots, vv_{p-1}\}$,   and without loss of generality, 
 assume that $A(v_{j})-\{v_jv\}\ne \emptyset$ for  $0\le j \le |A^{f}|-1$.
Denote the elements of $A^{f}$  by  $v_0v_0', v_1v_1', \ldots, v_{|A^{f}|-1}v_{|A^{f}|-1}'$.
 Note that $(v, v_j, v_j')$ is a directed 
    path for   $0\le j \le |A^{f}|-1$. Consider the following possibilities:
    
\smallskip

    (1) $p=1$: Then $\{vv_0\}$ is a trivial $(A, \emptyset, \emptyset)$-net  of   size  $1$.

   (2) $|A^{c}|\ge 1$  and $p=2$:  Let   $ww'$ be an    $A$-compatible arc and 
    $A^{c}:=\{ww'\}$.  Since  $ww'$ is adjacent to each arc of 
  $A$,    $\{vv_0\}$,  $\{vv_1, ww'\}$ is  an $(A, \emptyset, A^{c})$-net   of   size  $2$.  
 
   (3) $|A^{f}|= 3$ and $p=3$:  Then $\{vv_0, v_1v_1'\}$, $\{vv_1, v_2v_2'\}$ and 
  $\{vv_2, v_0v_0'\}$ form   an $(A, A^{f}, \emptyset)$-net  of size $3$.
   
   (4)  $|A^{f}|\ge 1$ and $|A^{c}|\ge1$ and  $p=3$:  Let   $ww'$ be an    $A$-compatible arc and 
    $A^{c}:=\{ww'\}$. Note that $ww'$ is adjacent to 
  each $vv_i$,  and $v_0v_0'$ is adjacent to  $vv_2$ in $X(D)$. So    $\{vv_0, ww'\}$, $\{vv_1, v_0v_0'\}$ and  $\{vv_2\}$ 
   form  an $(A,A^{f},A^{c})$-net  of size $3$. 
 
   (5)  $|A^{c}|\ge2$  and  $p=3$: Similar to case (4), $\{vv_0, ww'\}$, $\{vv_1, yy'\}$ and  $\{vv_2\}$ form  
    an $(A,A^{f},A^{c})$-net   of size $3$, where 
    $A^c$  contains two   $A$-compatible arcs  $yy'$ and $ww'$.
      
   (6)  $|A^{f}|+|A^{c}|\ge p-1$ and  $p\ge 4$:    Let $\b_j:= v_jv_j'$ for 
    $0 \le j \le |A^{f}|-1$.  Since   $|A^{c}|\ge p-1-|A^{f}|$,  we can choose   $p-1-|A^{f}|$ arcs from  $A^{c}$ and  name them 
    as $\b_{|A^{f}|}$,   $\b_{|A^{f}|+1}$,  $\ldots$,  $\b_{p-2}$. Define $B_j:= \{vv_{j}, \b_{j+1}\}$
    for $0\le j \le p-3$,  $B_{p-2}:= \{vv_{p-2}, \b_{0}\}$, and   $B_{p-1}:= \{vv_{p-1}\}$. 
    For   $0\le i< j \le p-2$,   observe that in $X(D)$, 
	  $vv_{j}\in B_{j}$  is adjacent to $\a_{i}$ 	if $i\ne j-1$; and   
     $vv_{i}\in B_{i}$  is adjacent to $\a_{j}$ if $i= j-1$,  where $\a_{j}\in B_j-\{vv_j\}$ and $\a_{i}\in B_{i}-\{vv_{i}\}$.
		Thus,   $B_{j}$ and  $B_{i}$ are adjacent. 
     In addition,  since $vv_{p-1}\in B_{p-1}$ is adjacent  in $X(D)$ to every $\b_{j}$,   $B_{p-1}$ is adjacent to   $B_j$ with $j\le p-2$.  Thus, 
      $B_0, \ldots, B_{p-1}$ form an $(A,A^{f},A^{c})$-net   of size $p$.    
\end{proof}

    Note that if $D$ contains an $(A,A^{f},A^{c})$-net,  then   $X(D)$ contains a $K_p$-minor  and $h(X(D))\ge p$. 
    
   A graph $G$ with chromatic number $k$ is called   {\em $k$-critical} if $\chi(H)<\chi(G)$ for every 
   proper subgraph $H$ of $G$. The following result is well known:
\begin{lemma}
\label{le7}
 Let $G$ be a   $k$-critical graph. Then
 \begin{itemize}
    \item[\rm (a)]$G$    has minimum degree   at least $k-1$,  when $k\ge 2$ {\em \cite{Dir53}};
    \item[\rm (b)]  no vertex-cut of  $G$  induces a clique when $k\ge 3$ and $G$ is noncomplete {\em  \cite{Dir52}}.
      \end{itemize}
 \end{lemma}

    Let $D$ be a  simple   digraph.  For each arc $uv \in A(D)$, define  $S_{D}(uv):=  d^+(u)+ d^+(v)-1$.

\begin{lemma}
\label{le6}     
  For a simple digraph  $D$, 
 \begin{eqnarray*}
   \sum_{uv\in A(D)} S_D(uv) &=&\sum_{v\in V(D)}d^{+}(v)(d(v)-1), 
     \end{eqnarray*}
   where $d(v)=d^{+}(v)+d^{-}(v)$.
\end{lemma}

     \begin{proof}
 \begin{align*}
  \sum_{uv\in A(D)} S_D(uv) &= \sum_{uv\in A(D)} (d^+(u)+ d^+(v)-1) \\
                         & =  \sum_{uv\in A(D)} d^+(u)+ \sum_{uv\in A(D)}  d^+(v) -   \sum_{uv\in A(D)}  1  \\
                         & =  \sum_{u \in V(D)} d^+(u)d^+(u)+ \sum_{v\in V(D)}  d^+(v)d^-(v) -   \sum_{u \in V(D)}   d^+(u)  \\ 
                       & =  \sum_{w \in V(D)} d^+(w)(d^+(w)+  d^-(w) - 1)  \\ 
                         & =  \sum_{w \in V(D)} d^+(w)(d(w) - 1).     \qedhere
			\end{align*}  
			\end{proof}

    \section{Proof of Theorem \ref{th3}}

	In this proof, we   assume that, for every pair of distinct vertices $u$ and $v$ of $D$, 
		there is at most one arc from $u$ to $v$ and at most one arc  from $v$ to $u$.
		  That is,    $A_D\{u,v\} \subseteq \{uv, vu\}$. That is because all the  arcs
   from $u$ to $v$  can be assigned the same colour and deleting arcs does not increase    $h(X(D))$.

Let $D$ be a digraph. An arc $uv$ of $D$ is called {\em redundant}  if  $A_{D}(u) \subseteq A_{D}\{u,v\}$
or  $A_{D}(v) \subseteq A_{D}\{u,v\}$.  Note that if $uv$ is redundant then so is  $vu$ if it exists.
Let  $D'$ be the digraph obtained from $D$ by deleting all  redundant arcs. 
Let $G$ be  the (simple) underlying undirected graph of $D'$.  We have the following claim:

\medskip
\medskip

{\bf Claim  1.}  $\chi (X(D))\leq   \chi(G)$.

\medskip
\medskip

{\em Proof.}  Since $G$ is  the underlying undirected graph of $D'$,      $V(G)=V(D')=V(D)$. 
Let 
$c: V(G) \rightarrow \{1,2, \ldots, \chi(G)\}$  be a $\chi(G)$-colouring  of $G$. 
 For each arc $uv\in A(D)$,  define $f(uv):=c(u)$. We now show that   $f$ is a $3$-arc colouring of $D$.
 For every pair of adjacent arcs $uv, xy \in A(D)$, we have that  $A_{D}\{u,x\}\ne \emptyset$ 
 (that is,  $u,x$ are adjacent),  and  both $uv$ and $xy$ are   not in $A_{D}\{u,x\}$. Thus, some    arc   between $u$ and $x$
   is not redundant,   and   $u$ and $x$ are adjacent in $G$. So,  $f(uv)=c(u)\ne c(x)=f(xy)$. 
  It follows that  $f$ is a $3$-arc colouring of $D$ and  $\chi (X(D))\leq   \chi(G)$.  \qed

 Hadwiger's conjecture  is true for $k$-chromatic graphs with 
$k\le   6$. So   assume  that $\chi (X(D))\geq 7$.
Let $k:= \chi(G)$ and  let $H$ be a   $k$-critical subgraph of $G$. By Lemma \ref{le7}(a), $\delta(H)\ge k-1$.

Let $F$ be an orientation of $H$  such  that  each arc $uv$ of $F$ inherits  the orientation of an arc in $A_{D}\{u,v\}$ 
and the number of out-degree 1 vertices in $F$ is minimized.   An arc   $xy\in A(D)$  is called {\em  potential}
if $xy \notin A(F)$.  In particular, every redundant arc is potential.
    $F$ has the following property: 

\medskip

{\bf Property A.}   If   $d^+_{F}(v)= 1$   and $A_{F}(v)=\{vw\}$, then 
  there exists one potential arc  $vz$ outgoing from $v$  in $D$ such that  $vz \ne vw$, and  $z\notin V(F)$ or  
 $d^+_{F}(z)\in \{0,2\}$.   
 
\medskip

{\em Proof.}  Since $vw$ is not redundant,    $A_{D}(v) \not\subseteq A_{D}\{v,w\}$.  Let  $vz \in A_{D}(v)- A_{D}\{v,w\}$.
 Then   $vz \ne vw$.   Since  $vw$ is the unique  outgoing arc from $v$ in $F$,   $vz$ is potential.  Suppose that
$z\in V(F)$.   Suppose first that   $v$ and $z$ are not adjacent  in $F$.  Then  each arc between  $v$ and $z$ 
in $D$  including $vz$  is redundant. Since $vw \in A(D)$,     $A_{D}(z) \subseteq A_{D}\{z,v\}$.  That is,  
 no arc  is outgoing at $z$ in $D$ except possibly   $zv$.  
Thus, $d^+_{F}(z)=  0$ as desired.      Suppose next  that   $v$ and $z$ are  adjacent  in $F$.    
 By the assumption that $d^+_{F}(v)=1$,  $A_{F}\{z, v\}=zv$.  If $d^+_{F}(z) \ne2$, let $F'$ be  obtained from $F$ by 
replacing $zv$ by $vz$. 
  Then $d_{F'}^+(z) \ne 1$,   $d_{F'}^+(v) =2$  and  the out-degree of 
every other vertex remains unchanged.  Hence  $F'$ is  an orientation of $H$ with less out-degree 1 vertices  than $F$,  
which is a contradiction.   \qed

   In addition,  for  each arc $xy$ of $F$,  by the definition of $D'$,  $A_{D}(y) \not\subseteq A_{D}\{x,y\}$. 
	That is,  there is  an arc other than $yx$  outgoing from $y$ (hence, $d^+_{D}(y)\ge 1$)  and  there is a directed path
	in $D$ of length 2 starting from the   arc $xy$,  even if  $d^+_{F}(y)=0$.	
 Note that $F$ is a simple digraph and   $d_{F}(v)=d_{F}^+(v)+d_{F}^-(v)=d_{H}(v)\ge k-1$  by Lemma \ref{le7}(a).

By Claim 1, it suffices to prove  that $h(X(D))\ge k$.

 Let  $v\in V(F)$ be a vertex  with maximum out-degree $\Delta^+_{F}(v)$.   
If   $\Delta^+_{F}(v)\ge k$, let $A\subseteq A_{F}(v)$ with $|A|=k$, and let $A^f$ be a maximal   $A$-feasible set.
  Then $|A^f|=k\ge  6$   since 
   there exists a directed path of length 2 starting from every arc of $A$. 
By Lemma \ref{le2}(6)  with $p=k$ and $q=0$, 
there exists an $(A,A^{f},\emptyset)$-net  of size  $k$. Thus,  $h(X(D))\ge k$, and  the result holds.

  Now assume  that $\Delta^+(F)\le k-1$.  
By  Lemma \ref{le6} and since $F$ has minimum degree at least $k-1$,  
 \begin{eqnarray} \label{eq2}
  \sum_{uv\in A(F)} S_{F}(uv) &=& \sum_{v\in V(F)}d^{+}_{F}(v)(d_{F}(v)-1)  
                             \ge  (k-2)\sum_{v\in V(F)}d^{+}_{F}(v)  
                               = (k-2)e(F),
  \end{eqnarray}
    where $e(F)$ is the number of arcs of $F$.

 If  $\sum_{uv\in A(F)} S_{F}(uv) = (k-2)e(F)$, then  $d_{H}(x)=d_{F}(x)=k-1$ for every $x\in V(F)$. Since $\chi(H)=k$,
 by Brooks' Theorem \cite{brook}, 
 $H \cong K_k$  and $F$ is a tournament. By Lemma \ref{le1},  
 $h(X(D))\ge h(X(F))\ge k$, the result follows.

Now assume   that    $\sum_{uv\in A(F)} S_{F}(uv)  >   (k-2)e(F)$.
   We call a vertex $v$ of $F$ {\em  special} if  $d_{F}^{+}(v)=k-2$ and   $d_{F}^{-}(v)=1$ and  
 $d_{F}^{+}(v')=0$ for each $vv'\in A_F(v)$.  Let $W$ be the set of all special vertices of $F$, 
 and let $W^{+}:=\{xy\in A(F) \mid  x\in W\}$.
 Let $F'$ be the  digraph   obtained from $F$ by deleting the arcs in $W^{+}$.   Then, for each
 vertex $v$ of $F'$ with $d^{+}_{F'} (v)=d_{F'} (v)-1=k-2$, the head of (at least) one arc $vv'\in A(F')$ is not 
 a sink in $F$; that is,  $d^{+}_{F} (v')\ge1$.  Since  this outgoing arc at $v'$  in $F$ is not redundant,    $|d^{+}_{D} (v')|\ge2$. 
 
 Denote by $Q$ the set of sinks of $F$.  Then each arc of  $W^{+}$ has its tail in $W$ and head in $Q$. 
 Note that   $W$ is independent in $F$, and  $W \cap Q=\emptyset$. By Lemma \ref{le6},
 \begin{eqnarray*} 
   (k-2)e(F)   & < & \sum_{uv\in A(F)} S_{F}(uv) \\
	                &=& \sum_{v\in V(F)}d^{+}_{F}(v)(d_{F}(v)-1)  \\
                              & =& \sum_{v\in V(F)-(W\cup Q)}d^{+}_{F}(v)(d_{F}(v)-1)
                          +\sum_{v\in Q}d^{+}_{F}(v)(d_{F}(v)-1)+ \sum_{v\in W}d^{+}_{F}(v)(d_{F}(v)-1)\\
                   & =&  \Big ( \sum_{v\in V(F')-(W\cup Q)}d^{+}_{F'}(v)(d_{F'}(v)-1) \Big ) +0+   (k-2)\Big (|W^+| +   \sum_{v\in W}d^{+}_{F'}(v)\Big ).
  \end{eqnarray*}

Since vertices in $W \cup Q$ have outdegree 0 in $F'$, 
 \begin{eqnarray*} 
    (k-2)e(F)   & < &  \Big (\sum_{v\in V(F')}d^{+}_{F'}(v)(d_{F'}(v)-1) \Big )+ |W^+| (k-2)   \\
          & = &     \Big (\sum_{uv\in A(F')} S_{F'}(uv) \Big ) + |W^+| (k-2).
  \end{eqnarray*}

 Thus $\sum_{uv\in A(F')} S_{F'}(uv) 
  >   (k-2) (e(F)-|W^+|) = (k-2)e(F')$. Let $uv$ be an arc  of $F'$
 with  maximum $S_{F'}(uv)$.  Thus, $S_{F}(uv)\ge S_{F'}(uv) \ge k-1$.
If $v\in W$, then $d_{F'}^+ (v)=0$ and  $d_{F'}^+ (u)\ge k$, which 
contradicts the assumption that $\Delta ^+ (F)\le k-1$.  Hence  $v\notin W$.

 Denote  $A_{F}(u)=\{uv, uu_1, uu_2, \ldots, uu_i\}$ and  $A_{F}(v)=\{vv_1, vv_2, \ldots, vv_j\}$,  where
$i+j= S_F(uv) \ge k-1$. 
 Set $T:=\{u_1,  u_2, \ldots, u_i\} \cap\{v_1, v_2, \ldots, v_j\}$.  
 Denote $N_1:=N_{F}(u)-\{v\}$ and 
    $N_2:=N_{F}(v)-\{u\}$.  Say $N_1=\{x_1,x_2,\ldots, x_r\}$, and $N_2 =\{y_1,y_2,\ldots, y_s\}$. 
		Since $F$ has minimum degree at least $k-1$,   both   $r$ and $s$ are at least $k-2$.
    
     Since  the arc   $A_{F}\{u, x_l\}$  is not redundant,   $A_{D}(x_l) \not \subseteq A_{D}\{u, x_l\}$. Thus,   for each $x_l\in N_1$,  
   to  arc $A_{F}\{u, x_l\}\in A(F)$ 
     we can associate an arc, denoted
    $\varphi(u, x_l)$,  which is chosen from   $A_{D}(x_l)-A_{D}\{u, x_l\}$.
    Similarly,   for each  $y_l\in N_2$,   associate an arc,  denoted
    $\varphi(v, y_l)$,  in   $A_{D}(y_l)-A_{D}\{v, y_l\}$  to  arc $A_{F}\{v, y_l\}\in A(F)$.

    Choose these arcs  $\varphi(u, x_l)$ and $\varphi(v, y_l)$  such that if  $\Sigma :=\cup_{l=1}^r  \varphi(u, x_l)$ 
       and  $\Pi:= \cup_{l=1}^s  \varphi(v, y_l)$ then  $t:= |\Sigma\cap \Pi|$  is  minimized.  
		We now prove   that, for each $ww'\in \Sigma\cap \Pi$, 
      $ww'$ is the unique arc outgoing from $w$ in $D$,     $A_{F}\{u, w\}=uw$,   $A_{F}\{v, w\}=vw$  and  $w'\notin\{u,v\}$. 
     Since  $ww'= \varphi(u, w) =\varphi(v, w)$,  we have $w'\notin\{u,v\}$.    Suppose that 
     $|A_{D}(w)|\ge 2$, and $ww''$ is an arc outgoing from $w$ other than $ww'$ in $D$. Then  at least  one
     of $u$ and $v$,  say  $u$,  is not equal to  $w''$. 
       Now set $\varphi(u, w) :=  ww''$ and keep   $\varphi(v, w) =  ww'$. Then  $|\Sigma\cap \Pi|$ is decreased.
     Thus,  $ww'$ is the unique arc outgoing from $w$ in $D$.  Since $A_{D}(w)=\{ww'\}$,  
     we have that  $A_{F}\{u, w\}=uw$ and    $A_{F}\{v, w\}=vw$.  
     
     Denote $\Sigma\cap \Pi =\{w_1w_1',w_2w_2', \ldots, w_tw_t'\}$.  Then $w_l\in T$ for each $l\in [1,t]$ 
     and $t \le |T|\le \min\{i,j\}$.
     Consider the following cases:
  
  \bigskip
 
 {\bf Case 1.}   $S_{F}(uv)\ge k$.
 
   \bigskip
 
    	In this case,   we construct an    $(A, A ^f, A^c)$-net $\AA$ and a  $(B, B^f, B^c)$-net  $\BB$,
			for some $A\subseteq A_F(u)-\{uv\}$ and $B\subseteq A_F(v)$, such that  
			$(A\cup A^f \cup A^c)\cap (B \cup B^f \cup \B^c)=\emptyset$.
			Since 	
				each branch set in $\AA$ contains an outgoing arc at $u$ other than $uv$, and each branch set in $\BB$ contains an outgoing arc 
		at $v$ other than 	$vu$,
	  each branch set in $\AA$ is adjacent in $X(D)$ to each branch set in $\BB$. 
	Since each branch set in $\AA$ is contained in $A\cup A^f \cup A^c$, and each branch set in $\BB$ is contained in
		$B \cup B^f \cup \B^c$,  no  branch set in $\AA$  intersects a  branch set in $\BB$.  
		  Hence $\AA \cup \BB$  defines a complete minor in $X(D)$   on  $|\AA|+  |\BB|$ vertices.  In most  cases we 
				construct  $\AA$  and $\BB$  such that  $|\AA|+  |\BB|\ge k$,  giving a   $K_k$-minor  in $X(D)$, as desired.  
	 Finally,   we always choose $A^c\subseteq \Sigma$ and $B^c\subseteq \Pi$
	in such a way that  $A^c \cap B^c=\emptyset$.

  Note that $i+j\ge k$. By the assumption that $\Delta^+(F)\le k-1$, we have  $1\le i\le k-2$
  and  $2\le j\le k-1$.  
  
 \medskip 
{\bf Case 1.1.}   $j=k-1$:   Then  $i\ge 1$.  Let $B:=A_{F}(v)$, and $B^f$  be a maximal 
  $B$-feasible set  in $D$. 
	For   $y_l\in N_2$, since $A_{D}\{y_l,v\}$ is not redundant,   $A^+_{D}(y_l)- A_{D}\{y_l,v\}\ne \emptyset$.
	Thus, 
	$|B^f| = |B|=k-1\ge 4$. By Lemma \ref{le2}(6) with $p=  |B^{f}|= k-1$ and $|B^{c}|=0$,
  there exists  in $D$ a $(B,B^{f},\emptyset)$-net   $\BB$
  of size $k-1$. Then $\BB  \cup \{\{uu_1\}\}$    forms the $k$ branch sets of a $K_k$-minor 
  in $X(D)$, since each branch set of   $\BB$ contains an outgoing arc at $v$  other than $vu$ and
		is thus  adjacent to $uu_1$ in $X(D)$ (since $vu \notin B$). 
   
  \medskip 
{\bf Case 1.2.}  $j\le k-2$:   
   Then $0 \le t\le k-2$.   Recall that  $t= |\Sigma\cap \Pi|\le |T|$.

	\medskip
  {\bf Case 1.2.1.}   $t=k-2\ge 3$:   Suppose first that   $\Sigma -  \Pi \ne \emptyset$. 
	Let $x_lx_l'\in \Sigma -\Pi$.  Since $|A_F(u)-\{uv\}|=i\ge t\ge 3$, there are distinct arcs $uu_{a}, uu_{b}$ 
	in $A_F(u)-\{uv\}$
	with  $x_l  \notin \{u_{a},  u_{b}\}$. Let
  $A:=\{uu_{a}, uu_{b}\}$.  Note that  $x_lx_l'$ is $A$-compatible.
	Then    $\AA := \{ \{uu_{a}\}$,  $\{uu_{b},x_lx_l'\}\}$ is an  $(A, \emptyset, \{x_lx_l'\})$-net   of size $2$.
Let $B$ be a set of    $k-2$ arcs in   $A_{F}(v)$.  Then $B^{f}:=\{\varphi(v,y):  vy\in B\}$ is a $B$-feasible set of  $k-2$ arcs in $\Pi$.
	By Lemma~\ref{le2}(6) with $p=|A^{f}|=k-2$ and $|A^{c}|=0$,  there is a  $(B, B^{f}, \emptyset)$-net  $\BB$  of size  $k-2$.
			Each branch set in $\AA$ contains an outgoing arc at $u$ other than $uv$, and each branch set in $\BB$ contains an outgoing arc 
		at $v$ other than 	$vu$.
		Thus each branch set in $\AA$ is adjacent in $X(D)$ to each branch set in $\BB$. Since  $x_lx_l'\notin \Pi$ and
		$B^f\subseteq  \Pi$,  we have $(A\cup \{x_lx_l'\})\cap (B \cup B^f)=\emptyset$. Thus,  no branch set in $\AA$ intersects
		a branch set in $\BB$.  Hence $\AA \cup \BB$ is    a $K_k$-minor  in $X(D)$.

		By symmetry and since $uv$ is not used in this case,  
   if  $ \Pi-  \Sigma \ne \emptyset$, then we  obtain  a $K_k$-minor  in $X(D)$.

 Now  assume that  $\Sigma =\Pi$. Then   $|\Sigma| =|\Pi| = t=k-2$. 
   Set $w_0:=v$ and $w_0':=w_1$.  
   For $0\le l \le t$,   let     $B_{l}:=\{uw_{l},w_{l+1}w_{l+1}'\}$, where subscripts
    are taken modulo $t+1$;  and  let $B_{t+1}:=\{vw_2\}$. 
	For   $0\le l< l'\le t$,   either $uw_l$ is adjacent to $w_{l'+1}w_{l'+1}'$ or  
	$uw_{l'}$ is adjacent to $w_{l+1}w_{l+1}'$. Thus   $B_{l}$ is adjacent to   $B_{l'}$.  
    Note  that  $vw_2\in B_{t+1}$ is adjacent to $w_1w_1'\in B_0$ and   $uw_{l}\in B_{l}$ 
    with $1\le l\le t$. Thus  $B_{t+1}$  is adjacent to every  $B_{l}$ 
    with $0\le l\le t$. Therefore,  $B_0, B_1, \ldots, B_{t+1}$  form the $t+2=k$   branch sets of a $K_k$-minor 
  in $X(D)$.
  
   \delete{ Let $B_0:=\{uv,w_1w_1'\}$,  $B_1:=\{uw_1,w_2w_2'\}$, $\ldots$
    $B_{t-1}:=\{uw_{t-1},w_tw_t'\}$,  $B_{t}:=\{uw_{t},vw_1\}$, and 
    }
 
	\medskip
  {\bf Case 1.2.2.}  $\lceil \frac{k}{2}\rceil\le t \le k-3$:   
    \delete{We shall construct a $u$-net of size $t$, and 
   a $v$-net of size   $k-t$.  To achieve this, we allocate the first  $k-1-t$ arcs of 
   $\Sigma\cap \Pi$ to the construction of  the desired  $v$-net, and the remaining  $t-(k-1-t)$  to that of 
   the  desired $u$-net.}
	  For $k-t\le l\le t$,  set    $\a_{l}:=w_{l}w_{l}'$.   Choose  $k-2-t$ arcs $\a_{t+1}, \a_{t+2}, \ldots, \a_{k-2}$   from $\Sigma-  \Pi$ 
    (which exist since $|\Sigma-\Pi|= r-t\ge k-2-t$).
  Denote $A:=\{uw_1,  uw_2, \ldots, uw_t\}$.  Then,    $\a_{l}$ is $A$-feasible  when $k-t\le l\le t$, 
	and  $\a_{l}$ is $A$-compatible   when $t+1\le l\le k-2$. 
	Let 
	$A^f:= \{\a_{k-t},\a_{k-t+1}, \ldots, \a_{t}\}$ and  $A^c:=\{\a_{t+1}, \a_{t+2}, \ldots, \a_{k-2}\}$. 
	Note that $A^f$  is  $A$-feasible and  $A^c$  is  $A$-compatible.  
	 By Lemma~\ref{le2}(6), there exists  an  $(A, A ^f, A^c)$-net  $\AA$ of size $t$ in  $X(D)$.

   \delete{
  The desired $u$-net is constructed as follows:  $B_1=\{uw_1\}$,
   $B_2=\{uw_2,\a_{k-t}\}$, $B_3=\{uw_3, \a_{k-t+1}\}$, $\ldots$, 
    $B_{2t+2-k}=\{uw_{2t+2-k}, \a_{t}\}$,  $B_{2t+3-k}=\{uw_{2t+3-k}, \a_{t+1}\}$,
   $\ldots$,
    $B_{t-2}=\{uw_{t-2}, \a_{k-4}\}$,
   $B_{t-1}=\{uw_{t-1},  \a_{k-3}\}$, and  $B_{t}=\{uw_{t}, \a_{k-2}\}$.  
  Note that  $B_1$  is adjacent to every $B_l$ with $l\ne 1$ since  $uw_1$ is adjacent to every  $\a_{l'}$ with  $l'\in [k-t, k-2]$.
 For distinct $l,l'\in [2,t]$,  either $uw_l$ is adjacent to  $\a_{k-t+2+l'}$   or 
  $uw_{l'}$ is adjacent to  $\a_{k-t+2+l}$  since  either $l\ne k-t+2+l'$ or  $l' \ne k-t+2+l$ (because $l\ne l'$).
   So,  $B_{l} = \{uw_l, \a_{k-t+2+l}\}$ and  $B_{l'}=\{uw_{l'}, \a_{k-t+2+l'}\}$ are adjacent. 
	Therefore, $B_1, B_2, \ldots, B_t$ form a $u$-net  of size $t$.
[{\bf another way to understand}: denote $A:=\{uw_1,  uw_2, \ldots, uw_t\}$, $A^f:= \{a_{k-t},\a_{k-t+1}, \ldots, \a_{t}\}$, and 
   $A^c:=\{\a_{t+1}, \a_{t+2}, \ldots, \a_{k-2}\}$. By Lemma \ref{le2}, a $u$-net of size $t$ can be constructed using $A\cup A^f \cup A^c$.]
   }
   
   Next, for $1\le l\le k-t-1$, set   $\b_{l}:=w_{l}w_{l}'$.
  Choose  $k-2-t$ arcs $\b_{k-t}$, $\b_{k-t+1}$, $\ldots, \b_{2k-2t-3}$ from $\Pi- \Sigma$ 
    (which exist since $|\Pi- \Sigma|= s-t\ge k-2-t$).
   Note that  $|\Sigma\cap \Pi|= t\ge k-t$  and $2k-2t-3\ge k-t$. 
  Let $B:=\{vw_1,  vw_2, \ldots, vw_{k-t}\}$. 
	Then   $\b_{l}$ is $B$-feasible  when $1\le l\le k-t-1$, 
	and  $\b_{l}$ is $B$-compatible   when $k-t\le l\le 2k-2t-3$. 
	Let 
	$B^f:= \{\b_{1},\b_{2}, \ldots, \b_{k-t-1}\}$, and 
   $B^c:=\{\b_{k-t}, \b_{k-t+1}, \ldots, \b_{2k-2t-3}\}$.   Note that $B^f$ is  $B$-feasible  and 
   $B^c$ is  $B$-compatible. If $t=k-3$,  then by  Lemma~\ref{le2}(4),  
	there exists  a  $(B, B^f, B^c)$-net  $\BB$ of size  $k-t$  in  $X(D)$.   
  Otherwise $t\le k-4$  and by 
 Lemma \ref{le2}(6) with $p=k-t\ge 4$ and $|A^{f}|=k-t-1$ and $|A^{c}|=k-t-2$, 
    there exists  a  $(B, B^f, B^c)$-net  $\BB$ of size  $k-t$  in  $X(D)$.

   \delete{
  A $v$-net of size $k-t$  can be  constructed   as follows: 
    $B_1=\{vw_1,  \b_2\}$,  $B_2=\{vw_2, \b_{3}\}$, 
   $B_3=\{vw_3, \b_{4}\}$, $\ldots$,  $B_{k-t-2}=\{vw_{k-t-2}, \b_{k-t-1}\}$,
    $B_{k-t-1}=\{vw_{k-t-1}, \b_{k-t}\}$,   $B_{k-t}=\{vw_{k-t}, \b_{k-t+1}\}$, 
    where $\b_{k-t+1}=\b_1$ if $k-t=2k-2t-3$. Similarly, $B_1, B_2, \ldots, B_{k-t}$ form 
    a $v$-net of size $k-t$.  Since each branch set of the $u$-net  (resp. $v$-net) formed above contains 
    an arc other than $uv$  (resp. $vu$)   outgoing from $u$  (resp. $v$),  the union of all these branch sets form the 
    $k$ branch sets of a $K_k$-minor in $X(D)$.
   }

	\medskip
  {\bf Case 1.2.3.}    $t\le \lceil \frac{k}{2}\rceil-1$: 
  Let $j':=k-i$.  Since $i+j= S_F(uv)\ge k$, we have  $j'\le j$. 
	
   If $t=0$, then  $\Sigma\cap \Pi =\emptyset$.   Let  $A:= \{uu_1,uu_2,\ldots, uu_i\}$. 
    Note that each arc in $\Sigma$ is either $A$-feasible or  $A$-compatible, and no two 
		arcs in  $\Sigma$ share a tail.  Let $A^f$ ($A^c$, respectively)  be the set of
		$A$-feasible  ($A$-compatible, respectively)    arcs in    $\Sigma$.  Then   $A^f$
		is $A$-feasible  and  $A^c$ is  $A$-compatible.  Note that  $|A^f|+|A^c| =|\Sigma|  \ge i$, and 
	   $A^c  \ne \emptyset$  if  $i\le 2$ ($\Sigma$ contains an $A$-compatible arc  since
	 $|\Sigma| =r \ge k-2 \ge 3$). 
	If $i\ge 3$, then by  Lemma~\ref{le2}(3) or Lemma~\ref{le2}(6) with $p=|A^{f}|=i$, there is an  
	$(A,A^f, \emptyset)$-net  $\AA$  of size $i$. 
	If $i\le 2$, then  $A^c  \ne \emptyset$ (since   $|\Sigma| =r \ge k-2 \ge 3>i$).  By 
		  Lemma~\ref{le2}(1) or  (2) with $p=|A^{f}|=i$ and $|A^{c}|\ge 1$,  there is an  
	$(A, \emptyset, A^c)$-net  $\AA$  of size $i$. 
	Similarly,  let   $B \subseteq A_{F}(v)$ with $|B|=j'$.   Let $B^f$ ($B^c$, respectively)  be the set of
		$B$-feasible  ($B$-compatible, respectively)    arcs in    $\Pi$.  Note that  $|B^f|+|B^c| =|\Pi| =s\ge k-2\ge j\ge j'$.
	As in the construction of $\AA$, 	by Lemma \ref{le2},  there exists a  $(B,B^f, B^c)$-net  $\BB$  of size  $j'$.   
	   $\AA\cup \BB$ forms  a 
    $k$ branch sets of a $K_k$-minor in $X(D)$.

		Suppose that   $t\ge 1$  and $j=k-2$.
   If $t=1$,  then let  	 $A$ be a subset of $A_{F}(u) -\{uv\}$ with $uw_1 \in A$ and 
	$|A|=2$.   Note that  $|\Sigma -\Pi|=r-t\ge k-3\ge 3$.  Then  at least one arc in 
	 $\Sigma -\Pi$  is $A$-compatible.     If $t\ge 2$,  then let  	 $A:=\{uw_1, uw_2\}$. Then   $|\Sigma -\Pi|=r-t\ge k-2 - \lceil \frac{k}{2}\rceil +1 = 
	\lfloor \frac{k}{2}\rfloor -1\ge 2$  because $k\ge 6$.  Again,   at least one arc in   $\Sigma -\Pi$ is $A$-compatible.
	 In both cases, by Lemma \ref{le2}(2), there exists an  $(A,\emptyset, A^c)$-net $\AA$  of size $2$,  where 
	 $A^c$ is the set of $A$-compatible arcs in $\Sigma- \Pi$.   
  Let  $B:=A_{F}(v)$.   Note that each arc in $\Pi$ is either $B$-feasible or  $B$-compatible, and no two 
		arcs in  $\Pi$ share a tail.   Let $B^f$ ($B^c$, respectively)  be the set of
		$B$-feasible  ($B$-compatible, respectively)    arcs in    $\Pi$.
	 Since $|\Pi|=s\ge k-2=j\ge 4$,  by  Lemma \ref{le2}(6),  there is  a  $(B,B^f, B^c)$-net  $\BB$   of size $j$.  Then  
   $\AA\cup \BB$ forms the
    $k$ branch sets of a $K_k$-minor in $X(D)$.

  Suppose now that   $t\ge 1$ and $j \le k-3$.   Note that $i\ge t$.   Consider two possibilities:  (i) $i=t$, and (ii) $i\ge t+1$. 
  If  $i=t$,   then $t=i\ge k-j\ge 3$.
  Let $A:=\{uu_1,uu_2,\ldots, uu_t\}= \{uw_1,uw_2,\ldots, uw_t\}$.  
  	Note that $|\Sigma-\Pi|=r-t\ge k-2-t\ge (2t+1)-2-t=t-1\ge 2$.
		Since $\Sigma-\Pi \ne \emptyset$,  
	 at least one	arc in $\Sigma-\Pi$  is 
  $A$-compatible.     Let $A^f$ ($A^c$, respectively)  be the set of
		$A$-feasible  ($A$-compatible, respectively)    arcs in    $\Sigma-\Pi$. 
  By Lemma \ref{le2}(2), (4) or (6),
	 there exists an  $(A,A^f, A^c)$-net  $\AA$  of size $i$.  
 Let   $B:=\{vv_1,vv_2,\ldots, vv_{j'}\}$.
 Let $B^f$ ($B^c$, respectively)  be the set of
		$B$-feasible  ($B$-compatible, respectively)    arcs in    $\Pi$.
		Note that $j'=k-t\ge k- \lceil \frac{k}{2}\rceil+1=\lfloor \frac{k}{2} \rfloor +1 \ge 3$  and
  $|\Pi|=s\ge k-2\ge j\ge j'$.  By Lemma \ref{le2},
	there is	a  $(B,B^f, B^c)$-net  $\BB$   of size $j'$.

  If $i\ge t+1$, then  $j'=k-i\le k-t-1$. 
 Let   $B:=\{vv_1,vv_2,\ldots, vv_{j'}\}$    be a subset of $A_F(v)$ with $vw_1\in B$.
 By the assumption that $j \le k-3$,  there is at least 
 one incoming arc other than $uv$ at $v$.  Thus,   at least   one  arc in $\Pi- \Sigma$   is   $B$-compatible.
Let $B^f$ ($B^c$, respectively)  be the set of
		$B$-feasible  ($B$-compatible, respectively)    arcs in   $\Pi- \Sigma$. 
Note that  $|\Pi- \Sigma|=s-t\ge k-t-2\ge j'-1$. By Lemma  \ref{le2}(2), (4) or (6),
	there is a  $(B,B^f, B^c)$-net  $\BB$   of size $j'$. 
Let  $A:=\{uu_1,uu_2,\ldots, uu_{i}\}$.   Let $A^f$ ($A^c$, respectively)  be the set of
		$A$-feasible  ($A$-compatible, respectively)    arcs in    $\Sigma$. 
   Since $|\Sigma|=r \ge k-2\ge i$,  by Lemma  \ref{le2}(2), (3) or (6), there exists an  $(A,A^f, A^c)$-net  $\AA$  of size   $i$.  
    
   In each case above,  $\AA\cup \BB$  forms  a $K_k$-minor in $X(D)$.

  \bigskip
 
 {\bf Case 2.}   $S(uv)=k-1$:  Then $i+j=k-1$.  
 
   \bigskip
 
  	In this case,   we construct an    $(A, A ^f, A^c)$-net $\AA$ and a  $(B, B^f, B^c)$-net  $\BB$ as in 
	Case 1, except that   $|\AA|+  |\BB|= k-1$.
	We then define one further  branch set $B_0$ that,  with  $\AA$  and $\BB$,  forms the desired     $K_k$-minor  in $X(D)$.

\medskip

     {\bf Case 2.1.}     $j=1$:    Then  $i=k-2$.  Let  $A:= A_{F}(u)-\{uv\}$.  
			Let $A^f$ ($A^c$, respectively)  be the set of
		$A$-feasible  ($A$-compatible)    arcs in  $\Sigma -   \Pi$.    Since 
		 $t\le \min\{i,j\}=1$ and $r\ge k-2$, we have  
		 $|\Sigma -   \Pi| \ge r-t\ge k-3$. 
		Since  $|A^f|+|A^c| = |\Sigma -   \Pi| \ge k-3$  and $i=k-2\ge5$, 
		by Lemma  \ref{le2}(6), there exists an  $(A,A^f, A^c)$-net  $\AA$  of size   $i$.  
By Property A,    
  there exists a potential   arc   $vz\ne vv_1$   outgoing from $v$
  in $D$, such that  $z\notin V(F)$ or    $d^+_{F}(z)\in \{0, 2\}$. 
	Clearly,  $z\ne u$ since   $d^+_{F}(u)=i+1>3$. 
	Let $B:=\{vv_1, vz\}$,  and  $\t$ be an arc in $\Pi-\Sigma$
	such that  $\t\ne \varphi(v,v_1)$  and  $\t\ne \varphi(v, z)$.   $\t$ exists because $|\Pi-\Sigma|=s-t\ge k-2-t\ge k-3\ge 3$.
	Then    $\BB:= \{\{vv_1\}, \{vz, \t\}\}$  is a $(B, \emptyset, \{\t\})$-net of size $2$.  
    Thus,  $\AA\cup \BB$       forms  a $K_k$-minor in $X(D)$.

\medskip
    {\bf Case 2.2.}       $2\le j \le k-3$:  Then $2\le i \le k-3$.  Let $U:=N_1\cap N_2$ be the 
    common neighbourhood of $u$ and $v$ in $F$. 
    Say $U=\{a_1,a_2,\ldots, a_{|U|}\}$.  Then $T\subseteq U$ and $t\le  |T|\le |U|$.
    Recall that $t=|\Sigma \cap\Pi|$.

	\medskip
	 {\bf Case 2.2.1.}   $t\ge 2$: 
	Let $A:= A_{F}(u)-\{uv\}$.   Since 
  $2\le t \le \min\{i,j\}$,   we have $i =k-1-j\le k-1-t$. 
    Since there is at least one incoming 
	arc at $u$ (because $i\le k-3$),  	at least one arc in  $\Sigma - \Pi$ is  $A$-compatible. 
	Let $A^f$ ($A^c$, respectively)  be the set of
		$A$-feasible  ($A$-compatible)    arcs in  $\Sigma -   \Pi$. 
	Note that $|A^f|+|A^c| = |\Sigma - \Pi|=r-t\ge k-2-t  \ge i-1$.
  By Lemma \ref{le2}(2), (4), (5) or (6),  there exists an  $(A,A^f, A^c)$-net  $\AA$  of size   $i$. 
	Let  $B:= A_{F}(v)$. Let  $B^f$ ($B^c$, respectively)  be the set of
		$B$-feasible  ($B$-compatible)    arcs in    $\Pi  - \Sigma$. 
	Similarly,  a  $(B,B^f, B^c)$-net  $\BB$   of size  $j$ exists (since $2\le i,j\le k-3$ and $uv$ is not in $\AA$).

		Let $B_0:=\{w_1w_1', w_2w_2', uv\}$.  Then  $B_0$
	induces a connected subgraph in $X(D)$ by noting that  $uv$ is adjacent to  both $w_1w_1'$ and $w_2w_2'$. 
  Each branch set of $\AA$ and 
    $\BB$ contains an arc outgoing from $u$ or $v$, which is adjacent to 
      $w_1w_1'$ or $w_2w_2'$. Thus  $B_0$ is adjacent to each branch set of  $\AA\cup \BB$. Hence  
			$\AA\cup \BB \cup \{B_0\}$   forms a $K_k$-minor in $X(D)$.

\medskip
   {\bf Case 2.2.2.}    $t\le 1$ and $U \cap N_{F}^{-}(v)\ne \emptyset$: That is, there is an arc $av$ in $F$
  for some vertex $a\in U$. If there exists an arc $a\bar{a}$  in $D$  with
  $\bar{a}\notin \{u, v\}$,  then let $B_0:=\{uv, a\bar{a}\}$.

  Suppose that there is no such  arc  $a\bar{a}$. That is, $A_{D}(a)\subseteq \{au, av\}$. 
 Clearly,  $av\in A_{D}(a)$. 
 Since 
  $A_{F}\{v,a\}$ is not  redundant in $F$, we have    $A_{D}(a)- A_{F}\{v,a\} \ne \emptyset$.
   Thus  $au\in A_{D}(a)$  and  $A_{D}(a)= \{au, av\}$.
  Let $\bar{a}$ be an in-neighbour  other than $u,v$ of $a$ in $F$. Then $A_{F}\{a, \bar{a}\}= \bar{a}a$. 
  Let $\bar{\bar{a}}\ne a$ be an out-neighbour of $\bar{a}$ in $F$. 
   Note that  $\bar{\bar{a}}$ exists since 
  $\bar{a}a$ is not redundant. Then, by the minimality of $|\Sigma \cap\Pi|$, 
 we have   $\bar{a}\bar{\bar{a}}\notin \Sigma \cap\Pi$.  Let $B_0:=\{uv, au, av, \bar{a}\bar{\bar{a}}\}$.  
  Then  $\max\{|B_0 \cap\Sigma|, |B_0 \cap\Pi|\}\le 2$ and $|B_0 \cap\Sigma|+ |B_0 \cap\Pi|\le 3$.
   
    Let $A:= A_{F}(u)-\{uv\}$  and  $B:= A_{F}(v)$. 
     We show that there  is  a net  $\AA$  at $u$ of size   $i$, and a net $\BB$  at $v$ of size   $j$,  
     such that 	$\AA\cup \BB \cup \{B_0\}$   forms a $K_k$-minor in $X(D)$.  
     
  First suppose that $3\le i,  j\le k-4$. If $|B_0 \cap\Sigma|\le 1$, 
 let $A^f$ ($A^c$, respectively)  be the set of
		$A$-feasible  ($A$-compatible)    arcs in  $\Sigma -\Pi -B_0$. 
		If $|B_0 \cap\Sigma|=2$, then $|B_0|=4$ and $\bar{a}\bar{\bar{a}}\in \Sigma \cap B_0$. 
		Thus, $\bar{a}$ is  a neighbour of $u$ in $F$.  Note that $\bar{a}a\notin \Sigma$ and   $\bar{a}a$ is $A$-feasible   or $A$-compatible.  
		Let $A^f$ ($A^c$, respectively)  be the set of
		$A$-feasible  ($A$-compatible)    arcs in  $(\Sigma -\Pi -B_0)\cup \{\bar{a}a\}$. 
	 In both cases,  $|A^f|+|A^c|  \ge r-t-1\ge k-2-2\ge   i$. 
  By Lemma \ref{le2}(3), (4), (5) or (6),  there exists an  $(A,A^f, A^c)$-net  $\AA$  of size   $i$. 
	  Let  $B^f$ ($B^c$, respectively)  be the set of
		$B$-feasible  ($B$-compatible)    arcs in    $\Pi  - (B_0\cup\{\bar{a}a\})$. 
	  Note that all arcs of $B_0\cup\{\bar{a}a\}$ except $uv$ are outgoing from at most two vertices
		(that is, $a$ and $\bar{a}$).
	 We have   $|B^f|+|B^c|  =  |\Pi  - (B_0\cup\{\bar{a}a\})|\ge s-2\ge k-4\ge   j$.	
	Similarly, by Lemma \ref{le2},  a  $(B,B^f, B^c)$-net  $\BB$   of size  $j$ exists.
 
  Next suppose  that  $i = k-3$ and $j=2$.  If $|B_0 \cap\Sigma|\le 1$, 
 let $A^f$ ($A^c$, respectively)  be the set of
		$A$-feasible  ($A$-compatible)    arcs in  $\Sigma  -B_0$. 
		If $|B_0 \cap\Sigma|=2$,  
		let $A^f$ ($A^c$, respectively)  be the set of
		$A$-feasible  ($A$-compatible)    arcs in  $(\Sigma -B_0)\cup \{\bar{a}a\}$, where $a, \bar{a}$ are as above. 
	 In both cases, we have   $|A^f|+|A^c|  \ge r-1\ge k-3  =i$. 
  By Lemma \ref{le2} (6),  there exists an  $(A,A^f, A^c)$-net  $\AA$  of size   $i$. 
  Let  $B^f$ ($B^c$, respectively)  be the set of
		$B$-feasible  ($B$-compatible)    arcs in    $\Pi  -\Sigma- (B_0\cup\{\bar{a}a\})$. 
		Since  $v$ has in $F$ at least $k-3\ge 4$ in-neighbours, one of which is not in $\{u, a, \bar{a}\}$. 
		Thus  $B^c\ne \emptyset$.
   By Lemma \ref{le2}(2),  a  $(B,B^f, B^c)$-net  $\BB$   of size  $2$ exists.

  Suppose that  $i = 2$ and $j=k-3$.  Let  $B^f$ ($B^c$, respectively)  be the set of
		$B$-feasible  ($B$-compatible)    arcs in    $\Pi  - B_0$.  Then 
		 $|B^f|+|B^c|  =  |\Pi  - B_0|\ge s-2\ge k-4=  j-1$.	 By Lemma \ref{le2}(6), there exists  a  $(B,B^f, B^c)$-net  $\BB$   of size  $j$.
  If $|B_0 \cap\Sigma|\le 1$, 
 let $A^f$ ($A^c$, respectively)  be the set of
		$A$-feasible  ($A$-compatible)    arcs in  $\Sigma -\Pi -B_0$. 
		If $|B_0 \cap\Sigma|=2$,  
		let $A^f$ ($A^c$, respectively)  be the set of
		$A$-feasible  ($A$-compatible)    arcs in  $(\Sigma -\Pi -B_0)\cup \{\bar{a}a\}$,  where $a, \bar{a}$ are as above.  
	 In both cases,  $|A^f|+|A^c|  \ge r-t-1\ge k-2-2\ge   3$. 
    Recall that $A=\{uu_1,uu_2\}$.  Note that  $|(A^f\cup A^c) - \{\varphi(u,u_1)\}|\ge2$. 
    Let $\t_1, \t_2$ be two arcs in $(A^f\cup A^c) - \{\varphi(u,u_1)\}$. Then, at least one arc, $\t_2$ say, of  $\t_1, \t_2$
    is not equal to $\varphi(u,u_2)$. Note that  $\t_2$ is adjacent to both $uu_1$ and $uu_2$, and $\t_1$ is adjacent to  $uu_1$ in  $X(D)$. 
    Let $\AA:=\{\{uu_1, \t_1\},\{uu_2, \t_2\}\}$.  Then,  $\AA$ is a   $(A,A^f, A^c)$-net     of size  $2$.

    In each case, 
     $B_0$
	induces a connected subgraph in $X(D)$.  And $uv\in B_0$ is adjacent to 
  each branch set of $\AA$,  and an arc outgoing from $a$ other than $av$   is adjacent to 
    each branch set of  $\BB$. Hence 
			$\AA\cup \BB \cup \{B_0\}$   forms a $K_k$-minor in $X(D)$.

\medskip
   {\bf Case 2.2.3.}    $t\le 1$ and $U \cap N_{F}^{-}(v)= \emptyset$  and $|U|\ge2$:
   That is, each arc  in $F$ between a vertex of $U$  and $v$ is outgoing at $v$. 
  Let $A:= A_{F}(u)-\{uv\}$  and  $B:= A_{F}(v)$.  We consider two
   situations.

  First suppose  that  $U$ is not independent in $F$.  That is, there is an arc  $\t$ in $F$ joining  two vertices in $U$. Say, $\t= a_1a_2$.
   Since $A_{F}\{u,a_2\}$ is not redundant, in $D$ there is an arc   $\g\ne a_2u$
    outgoing from $a_2$. (It may happen that   $\g\in \{a_2a_1, a_2v\}$.)
    Let $B_0:=\{uv, \t, \g\}$.  Since $uv$ is adjacent to both $\t$ and $\g$,    $B_0$ induces a connected subgraph
    in $X(D)$.  Note that $\max\{|B_0 \cap\Sigma|, |B_0 \cap\Pi|\}\le 2$.

 If $i\ge j$, then $j\le\frac{k-1}{2}\le k-4$. 
Let $A^f$ ($A^c$, respectively)  be the set of
		$A$-feasible  ($A$-compatible)    arcs in  $\Sigma  -B_0$;  and, let  $B^f$ ($B^c$, respectively)  be the set of
		$B$-feasible  ($B$-compatible)    arcs in    $\Pi -\Sigma - B_0$. 
	Then $|A^f|+|A^c| \ge r-2\ge k-2-2\ge   i-1$. 
  By Lemma \ref{le2}(6),  there exists an  $(A,A^f, A^c)$-net  $\AA$  of size   $i$. 
  Also,     $|B^f|+|B^c|  =  |\Pi  -\Sigma - B_0|\ge s-t-2\ge k-5\ge   j-1$.
  Note that there is at least one (in fact many) incoming arc $v_lv$ at $v$  with 
  $\varphi(v_l,v)\notin \Sigma\cup B_0$. Thus  $\varphi(v_l,v)\in B^c$ and   $|B^c|\ge 1$.
  By Lemma \ref{le2}(2), (4) or (6),   a  $(B,B^f, B^c)$-net  $\BB$   of size  $j$ exists.
  If $i\le j$, then $i\le\frac{k-1}{2}\le k-4$. 
	Now let $A^f$ ($A^c$, respectively)  be the set of
		$A$-feasible  ($A$-compatible)    arcs in  $\Sigma -\Pi  -B_0$;  and  let  $B^f$ ($B^c$, respectively)  be the set of
		$B$-feasible  ($B$-compatible)    arcs in    $\Pi - B_0$. 
	Similarly,   we obtain  an  $(A,A^f, A^c)$-net  $\AA$  of size   $i$ 
  and  a  $(B,B^f, B^c)$-net  $\BB$   of size  $j$.
  
  Since each arc outgoing from $u$ or $v$ is adjacent to $\t$ or  $\g$, each branch set of 	$\AA\cup \BB$
  is adjacent to $B_0$. Thus, 	$\AA\cup \BB \cup \{B_0\}$   forms a $K_k$-minor in $X(D)$.

 Next  suppose that  $U$ is independent in $F$. For each $a_{l}\in U$, if  in $D$ there is an arc  $a_{l}a_{l}'$
other than $a_{l}u$ or  $a_{l}v$, let $Q_{l} :=\{a_{l}a_{l}'\}$.  Otherwise, we have $A_{D}(a_{l})=\{a_{l}u, a_{l}v\}$.
Let $\bar{a_{l}}$ be an in-neighbour other than $u,v$ of $a_{l}$ in $F$. Then $A_{F}(\bar{a_{l}}, a_{l})= \bar{a_{l}}a_{l}$.
Let $\bar{\bar{a_{l}}} \ne  a_{l}$ be an out-neighbour of  $\bar{a_{l}}$ in $F$.    Let $Q_{l} :=\{a_{l}u, a_{l}v,  \bar{a_{l}}\bar{\bar{a_{l}}}\}$.
Let $a_{l}, a_{m}$ be distinct vertices in $U$ such that $w_1 \in \{a_{l}, a_{m}\}$  when $t=1$   and  $|Q_{l}\cup Q_{m}|$ is minimised. 
   Let $B_0:=\{uv\} \cup Q_{l}\cup Q_{m}$. Note that in $X(D)$
each of the subgraphs induced on  	$Q_{l}$ and  $Q_{m}$  is connected  and  adjacent to  $uv$,  $B_0$ induces a connected subgraph.

    Note that for each $p\in \{l,m\}$,    $|Q_{p}  \cap \Sigma| \le 2$  and  $|Q_{p}  \cap \Pi| \le 2$.   If $|Q_{p}  \cap \Sigma|= 2$,
		then  $Q_{p} :=\{a_{p}u, a_{p}v,  \bar{a_{p}}\bar{\bar{a_{p}}}\}$ and  $\bar{a_{p}}\bar{\bar{a_{p}}} \in \Sigma$ and
			$\bar{a_{p}}$ is adjacent to $u$ (but not $v$ because $U$ is independent) in $F$. 
			Thus  $\bar{a_{p}}a_{p}$ is  $A$-feasible  ($A$-compatible) if  $\bar{a_{p}}\bar{\bar{a_{p}}}$ is  $A$-feasible  ($A$-compatible).
			Let $\Sigma'$  be obtained from $\Sigma$  by replacing $\bar{a_{p}}\bar{\bar{a_{p}}}$ with $\bar{a_{p}}a_{p}$.  Then  $|Q_{p}  \cap \Sigma'| \le 1$   and  $|B_{0}  \cap \Sigma'| \le 2$. In addition, each element in $\Sigma'$  is $A$-feasible or $A$-compatible, and no two share a tail. 
    Similarly, we can obtain $\Pi'$ such that  each of its elements    is $A$-feasible or $A$-compatible,   no two elements share a tail and  $|B_{0}  \cap \Pi'| \le 2$.
  
  Let $A^f$ ($A^c$, respectively)  be the set of
		$A$-feasible  ($A$-compatible)    arcs in  $\Sigma'  -B_0$;  and let  $B^f$ ($B^c$, respectively)  be the set of
		$B$-feasible  ($B$-compatible)    arcs in    $\Pi' -  B_0$. 
	Then, $|A^f|+|A^c|  \ge r-2\ge k-2-2\ge   i-1$. 
  Also,    $|B^f|+|B^c|  =  |\Pi' - B_0|\ge s-2\ge k-4\ge   j-1$.
  When $i=2$, since $|A^f|+|A^c|\ge k-4\ge 3$, we have   $A^c\ne\emptyset$.  Analogously,  we have that 
   $B^c\ne\emptyset$  when $j=2$.
   By Lemma \ref{le2}(2)-(6),  there exist  an  $(A,A^f, A^c)$-net  $\AA$  of size   $i$ and  a  $(B,B^f, B^c)$-net  $\BB$   of size  $j$.

  Since each arc outgoing from $u$ or $v$ is adjacent to an arc in  $Q_{l}$ or    $Q_{m}$, each branch set of 	$\AA\cup \BB$
  is adjacent to $B_0$. Thus, 	$\AA\cup \BB \cup \{B_0\}$   forms a $K_k$-minor in $X(D)$.

\medskip
   {\bf Case 2.2.4.}    $U \cap N_{F}^{-}(v)= \emptyset$  and $|U|\le1$ (hence $t\le1$):
   That is, $u$ and $v$ share at most one neighbour $a_1$ in $F$. If $a_1$ exists, the arc 
   between $a_1$ and $v$ in $F$ is $va_1$. 
  Let $A:= A_{F}(u)-\{uv\}$  and  $B:= A_{F}(v)$.

	Since $\delta (F)\ge k-1$ and $j \le k-3$,  
	$v$ has at least  $k-1-j \ge 2$ in-neighbours in $F$.  Say,  $N_F^-(v)=\{u,  y_{j+1}, y_{j+2}, \ldots,
	 y_{k-2}\}$.    Note that  $N_F^-(v)-\{u\} \ne \emptyset$. Recall that  $N_F^+(v)=\{v_{1}, v_{2}, \ldots,
	 v_{j}\}$.

			Let $\bar{H}$ be obtained from $H$ by deleting vertices  in $U \cup \{u, v\}$.    By Lemma \ref{le7}(b),  
   $\bar{H}$ is connected.  
Let $P_0:=(z_1,z_2,\ldots, z_m)$ be a shortest path
  in $\bar{H}$   between $N_F(u)-(\{v\} \cup U)$
and $N_F^-(v)-(\{u\}\cup U)$, where $m\ge2$ (because $u$ and $v$ share no common neighbour in  $\bar{H}$), 
 $z_1\in  N_F(u)-(\{v\}\cup U)$ and  $z_m\in N_F^-(v)-(\{u\}\cup U)$. 
Then
 each internal vertex  of $P_0$
is not  adjacent to $u$   in $F$. 
			
		\delete{	Let $\bar{H}:= H- \{u,v\}$.    By Lemma \ref{le7}(b),  
   $\bar{H}$ is connected.  
Let $P_0:=(z_1,z_2,\ldots, z_l)$ be a shortest path
  in $\bar{H}$   between $N_F(u)-\{v\}$
and $N_F^-(v)-\{u\}$, where $l\ge1$,  $z_1\in  N_F(u)-\{v\}$ and  $z_l\in N_F^-(v)-\{u\}$. 
Then, 
 each internal vertex $z_{l'}$  (that is,  $l'\ne 1$ or $l$) of $P$
is not  adjacent to $u$   in $F$.   
 } 

If  $|V(P_0)\cap  N_F(v)|=1$,  then     $z_m$ is the only
neighbour of $v$   in $F$ which is   on $P_0$.  Let $P:=P_0$ and set $z_l:=z_m$. 
 If  $|V(P_0)\cap  N_F(v)|\ge 2$, 
let $P= (z_1,z_2,\ldots, z_l)$ be the subpath  of  $P_0$  such that  $z_l\in N_F(v)$  and   $|V(P)\cap  N_F(v)|= 2$ .

 We shall  construct  a   branch set  $P'$    consisting of  arcs    alongside $P$.  
Let $z_0=u$ and $z_{l+1}=v$.  

For    $1\le g\le l$,  we associate to   $z_{g}$  the set $Q_{g}$  of  arcs  as 
follows.  If $A_{D}(z_{g})-  (A_{D}\{z_{g-1}, z_{g}\} \cup  A_{D}\{z_{g}, z_{g+1}\})\ne \emptyset$,
   then let  $Q_{g}$  be a singleton set that  contains exactly one  arc,  say, 
 $z_{g}z_{g}'  \in A_{D}(z_{g})-  (A_{D}\{z_{g-1}, z_{g}\} \cup  A_{D}\{z_{g}, z_{g+1}\})$.
Otherwise, $A_{D}(z_{g})-  (A_{D}\{z_{g-1}, z_{g}\} \cup  A_{D}\{z_{g}, z_{g+1}\})= \emptyset$.  
Since  the arc  $A_{F}\{z_{g}, z_{g+1}\}\in A(F)$  is not redundant,   
  $z_{g}z_{g-1} \in A_{D}(z_{g})$.   
	Similarly, 
	  $z_{g}z_{g+1} \in A_{D}(z_{g})$ since  $A_{F}\{z_{g-1}, z_{g}\}\in A(F)$   is not redundant. 
 Let  $\bar{z}_{g}$  be an in-neighbour of $z_{g}$ in $F$. Then  $\bar{z}_{g}z_{g}\in A(F)$. 
 Let  $\bar{z}_{g}\bar{\bar{{z}}}_{g}$  with   $\bar{\bar{{z}}}_{g}  \ne z_{g}$  be an arc outgoing from  $\bar{z}_{g}$ in $D$
 (which  exists because $\bar{z}_{g}z_{g}$  is not redundant). 
 Set  $Q_{g}:=\{z_{g}z_{g-1}, z_{g}z_{g+1}, \bar{z}_{g}\bar{\bar{{z}}}_{g}\}$.
 Note that    $Q_{g}$  induces a connected subgraph in $X(D)$ since  $\bar{z}_{g}\bar{\bar{{z}}}_{g}$ 
  is adjacent to  both $z_{g}z_{g-1}$ 
and $z_{g}z_{g+1}$.

In the case where $V(P)\cap  N_F(v)=\{z_p,z_l\}$ ($p<l$)   and $Q_p= \{z_pv\}$,  we slightly modify  $Q_p$
as $\{z_pv,\g\}$,  where $\g\in A_{D}(z_p)-\{z_pv\}$ (which  exists because $A_{F}(z_p, v)$  is not redundant).

Let $P':= \cup_{g=1}^lQ_{g}$. 
Then,  for    $1\le  g  \le l-1$,  since   $Q_{g}$    contains an arc
  outgoing from $z_{g}$ other than $z_{g}z_{g+1}$ and   $Q_{g+1}$  contains an arc
  outgoing from $z_{g+1}$ other than $z_{g+1}z_{g}$,  each  $Q_{g}$ is adjacent to  $Q_{g+1}$ in $X(D)$.
  Thus,
 $P'$ induces a connected 
subgraph in $X(D)$.  We call $P'$   a {\em parallel set} of $P$.

Let $\Sigma$  and $\Pi$ be as above.  We have the following claim:

\medskip
\medskip

{\bf Claim  2.}   (a) There is a set $\Sigma'$  such that  $|\Sigma'|\ge |\Sigma|-1$  and $P' \cap \Sigma'=\emptyset$, and 
    each  element  of which  is $A$-feasible or $A$-compatible and no two elements share a tail;\\
  (b)  There is a set $\Pi'$  such that  $|\Pi'|\ge |\Pi|-2$  and $P' \cap \Pi'=\emptyset$, and 
    each  element  of which  is $B$-feasible or $B$-compatible and no two elements share a tail.

\medskip
\medskip

{\em Proof.}   (a)
Initially, set $\Sigma':= \Sigma-P'$. Clearly, all properties except  $|\Sigma'|\ge |\Sigma|-1$ in (a) are satisfied.
 If $|P' \cap \Sigma|\le  1$, then we are done. Suppose that $|P' \cap \Sigma|\ge 2$. 
 Since $P_0$ is  a shortest path  in $\bar{H}$   between $N_F(u)-(\{v\}\cup U)$
and $N_F^-(v)-(\{u\}\cup U)$, each vertex  $z_{g}$ on $P$ with $g\ge 3$ is not adjacent to a  vertex of  $N_F(u)-(\{v\}\cup U)$.
Thus,   $Q_{g} \cap \Sigma =\emptyset$ for each $g\ge 3$. 
We now consider $g=2$.  Since $z_{2}$ is not
 adjacent to $u$ in $\bar{H}$, we have $|Q_{2} \cap \Sigma| \le 1$ and if  $|Q_{2} \cap \Sigma| = 1$ then
$|Q_{2}|  = 3$ and  $Q_{2}:=\{z_{2}z_{1}$, $z_{2}z_{3}, \bar{z}_{2}\bar{\bar{z}}_{2}\}$,
  where $\bar{z}_{2}$  is an in-neighbour of $z_{2}$ in $F$. 
Since  $z_{2}$ is not adjacent to $u$, $Q_{2} \cap \Sigma= \{\bar{z}_{2}\bar{\bar{z}}_{2}\}$,
	which means that $\bar{z}_{2}$ is adjacent to $u$ in $F$  and $\varphi(u, \bar{z}_{2}) =\bar{z}_{2}\bar{\bar{z}}_{2}$.
	In this case,  update  
	$\Sigma':= \Sigma' \cup\{\bar{z}_{2}z_2\}$.  Note that  $\bar{z}_{2}z_2$ is  $A$-feasible or $A$-compatible.

	If $|Q_{1} \cap \Sigma| \le 1$,   then 	$\Sigma'$  is the  desired  set.
	Suppose that $|Q_{1} \cap \Sigma| =2$. 
 Let  $Q_{1}:=\{z_{1}u$, $z_{1}z_{2}, \bar{z}_{1}\bar{\bar{z}}_{1}\}$,
  where $\bar{z}_{1}$  is an in-neighbour of $z_{1}$ in $F$. 
	Then, $Q_{1} \cap \Sigma= \{z_{1}z_{2},   \bar{z}_{1}\bar{\bar{z}}_{1}\}$,
	which means  $\varphi(u, z_{1})=   z_{1}z_{2}$ and  $\varphi(u, \bar{z}_{1})=    \bar{z}_{1}\bar{\bar{z}}_{1}$.
	Note that  $\bar{z}_{1}z_1$ is  $A$-feasible or $A$-compatible. By adding  $\bar{z}_{1}z_1$  into	$\Sigma'$, 
	we get that $|Q_1 \cap \Sigma'|\le  1$. Then  $|\Sigma'|\ge |\Sigma|-1$, as desired.

	(b)
	   Initially, set $\Pi':= \Pi-P'$. 
Recall that $P$ contains at most two neighbours, $z_{g_1}$ and  $z_{g_2}$ say,  of $v$.  
Let $\g$ be an    arc  in $\Pi \cap P'$ 
such that there is a 
$Q_g$ containing $\g$ (there may be more than one $Q_g$ containing $\g$) and  
$g\notin \{g_1, g_2\}$. 
Since  $z_{g}$ is not
 adjacent to $v$ in $\bar{H}$, we have  $|Q_g|=3$  and
	 $Q_{g}=\{z_{g}z_{g-1}, z_{g}z_{g+1}, \bar{z}_{g}\bar{\bar{{z}}}_{g}\}$, 
	where 
	   $\bar{z}_{g}$  is an in-neighbour of $z_{g}$ in $F$  and  $\bar{z}_{g}\bar{\bar{{z}}}_{g} \ne  \bar{{z}}_{g}z_{g}$
		is an arc outgoing from  $\bar{z}_{g}$ in $D$. Further,  $\bar{z}_{g}$  is a neighbour of $v$ in $F$ 
  and $\varphi(v, \bar{z}_{g}) = \bar{z}_{g}\bar{\bar{{z}}}_{g}$.
	Note that $\bar{{z}}_{g}z_{g} \notin \Pi$ is $B$-feasible or $B$-compatible. Now  update   $\Pi'$ by  adding  $\bar{{z}}_{g}z_{g}$.
	 That is,   $\Pi':= \Pi' \cup \{\bar{{z}}_{g}z_{g}\}$. 
	By repeating this procedure  for all such $\g$,  we obtain a  $\Pi'$ with the same size as
	$\Pi - (Q_{g_1} \cup Q_{g_2})$.

	For each $g\in\{g_1,  g_2\}$, if  $|\Pi \cap Q_g|=2$,  we will add a $B$-feasible or $B$-compatible arc into $\Pi'$. 
	Then   $|\Pi'|\ge |\Pi|-2$, as desired. 
 Suppose that $|\Pi' \cap Q_g|=2$
	for some $g\in\{g_1,  g_2\}$.  Then  $Q_{g}=\{z_{g}z_{g-1}, z_{g}z_{g+1}, \bar{z}_{g}\bar{\bar{{z}}}_{g}\}$, 
	where  $\bar{z}_{g}$ is an in-neighbour of $z_{g}$ in $F$  and  $\bar{z}_{g}\bar{\bar{{z}}}_{g} \ne  \bar{{z}}_{g}z_{g}$
		is an arc outgoing from  $\bar{z}_{g}$ in $D$. And,  $\bar{z}_{g}$  is a neighbour of $v$ in $F$ 
  with $\varphi(v, \bar{z}_{g}) = \bar{z}_{g}\bar{\bar{{z}}}_{g}$. Note that $\bar{{z}}_{g}z_{g} \notin \Pi$ is $B$-feasible or $B$-compatible. 
	Set  $\Pi':= \Pi' \cup \{\bar{{z}}_{g}z_{g}\}$.  Then  $|\Pi'|\ge |\Pi|-2$.
	Consequently, we get the desired $\Pi'$.    \qed

	\delete{
		Then  $\bar{z}_{2}\bar{z}_{2}' \in \Sigma$, 
  $\bar{z}_{2}\in N_{F}(u)$ and  $\varphi (u, \bar{z}_{2})=  \bar{z}_{2}\bar{z}_{2}'$. 
  Note that  $\bar{z}_{2}z_2$ is outgoing from  $\bar{z}_{2}$ to $z_2$.  
	If   $\bar{z}_{2}z_2 \notin \Pi$,  update $\Sigma$
	by resetting   $\varphi (u, \bar{z}_{2}):= \bar{z}_{2}z_2$, and denote the resultant set again by $\Sigma$.
	If   $\bar{z}_{2}z_2 \in \Pi$,   then  $\bar{z}_{2}$ is adjacent to $v$   and    $\varphi (v, \bar{z}_{2}) = \bar{z}_{2}z_2$.  
	Update  $\Pi$  by resetting    $\varphi (v, \bar{z}_{2}):= \bar{z}_{2}\bar{z}_{2}'$,  and denote it again by $\Pi$. 
	
	$\Sigma$
	by resetting   $\varphi (u, \bar{z}_{2}):= \bar{z}_{2}z_2$, and denote the resultant set again by $\Sigma$.

	Then,   $P(z_{2}) \cap \Sigma =\emptyset$.
Note that  $|P(z_{1}) \cap \Sigma| \le 1$. Therefore,  $|P' \cap \Sigma|\le1$.

  Similarly, we can get  $\Pi$,  $P'$   with $|P' \cap \Pi|\le |V(P)\cap N_F(v)|$. 

,  $\bar{z}_{2}\notin N_{F}(v)$ (since $P$ is non-trivial)  

 and  
$|P' \cap \Pi|\le  |V(P)\cap N_F(v)|$ (note that  $|V(P)\cap N_F(v)|=1$ or $2$). 
,    $\Pi$ and $P'$ 
}

 Let $B_0:=\{uv\} \cup P'$.   Then  $B_0$ induces a connected subgraph in $X(D)$ since 
$uv$ is adjacent to $Q_{1}$.

 Next we show that there exists  a  net  of size $i$  at  $u$  and a  net  of size $j$ at   $v$
such that none of their  branch sets  intersects  $B_0$.  

  If  $j=2$ (hence $i = k-3$),   then at least one arc, say   $\g$, in  $\Pi'-\Sigma'$
is  $B$-compatible (since there are more incoming arcs at $v$).  Let   $B^c:=\{\g\}$.  
Since $|\Pi'-\Sigma'| \ge s-2-1\ge k-5 \ge j=2$, 
	by Lemma \ref{le2}(2),  there  exists a $(B, \emptyset, \B^c)$-net $\BB$ of size $j=2$. 
	Similarly,  	let $A^f$ ($A^c$, respectively)  be the set of
		$A$-feasible  ($A$-compatible, respectively)    arcs in  $\Sigma'$.
	Note that   $|\Sigma'|\ge r-1 \ge k-3=i\ge 4$.  
  By Lemma \ref{le2}(6), there exists an  $(A,A^f, A^c)$-net  $\AA$  of size   $i$.

  Suppose that $3\le j \le  k-3$  (hence  $2 \le i \le k-4$).  
		Let $B^f$ (respectively, $B^c$)  be the set of
		$B$-feasible  ($B$-compatible)    arcs in $\Pi'$. 
	Since    $|\Pi'|\ge s-2\ge k-4 \ge j-1$ and $B^c\ne \emptyset$  when $j=3$,
	 by Lemma \ref{le2}(4) or (6),  there  exists a $(B, B^f, \B^c)$-net $\BB$ of   size $j$.	
		Let $A^f$ ($A^c$, respectively)  be the set of
		$A$-feasible  ($A$-compatible, respectively)    arcs in  $\Sigma'-   \Pi'$.   We now show that there exists  a  net  of size $i$  at  $u$.  	
			  If  $i\ge 3$,  
				then 
	$|\Sigma'- \Pi'|\ge r-1-1 \ge k-4\ge i \ge 3$.   By Lemma \ref{le2}(3) or (6), there exists an  $(A,A^f, A^c)$-net  $\AA$  of size   $i$. 
	 Suppose that $i=2$.
	 Note that $|\Sigma'- \Pi'|  \ge k-4\ge 3$ (because $k\ge 7$) and  
	 there are at least three incoming arcs at $u$ in $F$.
	  $\Sigma'-   \Pi'$	contains at least two  $A$-compatible arcs, say,  $\gamma_1$ and  $\gamma_2$.
		Let  $\AA:= \{\{uu_1, \gamma_1\},  \{uu_2, \gamma_2\}\}$. Then $\AA$  is a net  of size $2$  at  $u$.

Since  each element of  $\AA$ constructed above  contains an arc  $xx'$,  which is  outgoing from  a neighbour $x\ne v$  of $u$ and  $x'\ne u$, 
 each element of  $\AA$  is adjacent to   $B_0$ because  $uv\in B_0$ is adjacent to each    $xx'$.    
 Note that  $|V(P)\cap N_F(v)| \in \{1, 2\}$. In the case when $|V(P)\cap N_F(v)|=1$,  $P'$ contains an arc  $yy'$,  which is outgoing from an in-neighbour $y\ne u$ of   $v$ and $y'\ne v$.  Since such a $yy'$ is adjacent to every arc of $A_F(v)$,  it is adjacent to every element of 
  $\BB$ constructed above.  In the case when $|V(P)\cap N_F(v)|=2$,   $P'$ contains two arcs  $\a$ and $\b$, each of them
 is outgoing from a neighbour  of $v$ other than $u$ and heading to  a vertex other than $v$. Then each arc of $A_F(v)$ is 
adjacent to either $\a$ or $\b$. So  every element of   $\BB$ is adjacent to $P'\subseteq B_0$.
Therefore,  $\{B_0\}\cup \AA \cup \BB$  
 forms   a $K_{k}$-minor in $X(D)$.

 \medskip
   {\bf Case 2.3.}      $j=k-2$:  Then  $i=1$.
  Suppose first that     $d_{F}^-(v)=1$;  that is,  $uv$ is the only incoming arc at $v$ and  $d_{F}(v)=k-1$. 
	Since $v$ is not special,  one out-neighbour $v'$ of $v$ in  $F$  is not a sink.    
 Now  consider  the arc  $vv'$.  If $d^+_{F}(v')\ge 2$, then  $S_{F}(vv')= d_{F}^+(v) + d^+_{F}(v')-1\ge  k-2 +2-1=k-1$. 
 This is a special case of Case 2.2 and thus can be  treated similarly.  
    If   $d^+_{F}(v')=1$,  then by Property A,    one  potential  arc $v'v''$ ($\ne v'v$)  is  outgoing from $v'$ in $D$  but  
  not present  in   $F$ (since  $d_{F}^+(v)=1$).  Let $F'$ be obtained from $F$ by  adding     $v'v''$.
  Again we have $S_{F'}(vv')= d_{F'}^+(v) + d^+_{F'}(v')-1\ge  k-2 +2-1=k-1$, and this  can also be  treated similarly.
	Suppose  next that     $d_{F}^-(v)\ge 2$. Then  $t\le 1$. This case can be dealt with by a similar way as in    Cases 2.2.3  or   2.2.4.

  \delete{
   As in  Case 2.2,   a  net at $v'$ of size 2,  a net at 
	$v$ of size $k-3$ and an additional 
	branch set (from certain shortest path between $N_F(v)-\{v'\}$  and $N_F(v')-\{v\}$)  can be constructed and the union 
	of them forms  a  $K_{k}$-minor in $X(D)$.
}

	\delete{		
			Suppose  next that     $d_{F}^-(v)\ge 2$.  Say,   $N^-_F(v)=\{u, y_1, \ldots, y_l\}$, where $l\ge 1$.  Then, $d_{F}(v)\ge k$. 
Recall that $N_F^+(v)=\{v_1, v_2, \ldots, v_{k-2}\}$.

Note that $t\le 1$ in this case. By a similar way as in   Case 2.2.3  or  Case 2.2.4,   
we can obtain a  net at  $u$ of size $1$  using arcs of $\{uu_1\}$ and $\Sigma' -P'$, a  net at  $v$ of size $k-2$  using arcs of $A_F(v)$ and
 $\Pi' -P'$,  and an additional    branch set $B_0:=\{uv\}\cup P'$,  where $P'$ is  a parallel set  of $P$  such that $|P' \cap \Sigma'|\le1$ and  
$|P' \cap \Pi'|\le 2$.  The union of them forms a  $K_{k}$-minor in $X(D)$.
}

\medskip
   {\bf Case 2.4.}      $j=k-1$:   Then  $i=0$,   which implies    $d_{F}^+(u) =1$.   By Property A, there   exists a potential 
	arc $uz\ne uv$  in $D$. 
Then   $\AA:= \{ \{uz\}\}$    is a  $(\{uz\}, \emptyset, \emptyset)$-net.   
  	Let  $B:= A_{F}(v)$. Let  $B^f$ ($B^c$, respectively)  be the set of
		$B$-feasible  ($B$-compatible, respectively)    arcs in    $\Pi$. 
 	By Lemma  \ref{le2}(6),   a  $(B,B^f, B^c)$-net  $\BB$   of size  $j$ exists.
	  It is not hard to see that  $\AA \cup \BB$  forms   a  $K_{k}$-minor in $X(D)$.

 This completes the proof of Theorem 1. \qed

\vskip 1pc 

\end{document}